\def\ZZ         {{\mathbb Z}}

\def\CC         {{\mathbb C}}

\def\LL         {{\mathbb L}}  
\def\QQ         {{\mathbb Q}}
\def\PP         {{\mathbb P}}
\def\NN         {{\mathbb N}}

\def\dim        {{\rm dim}}

\def\codim        {{\rm codim}}

\documentclass[12pt]{amsart}
\usepackage{amsmath,amscd,amssymb,amsfonts,latexsym}
\usepackage[all]{xy}

\usepackage[latin1]{inputenc}

\usepackage[T1]{fontenc}

\usepackage[extra]{tipa}

\usepackage{eucal} 
\usepackage{cmmib57} 
\usepackage{amsbsy}
\usepackage{amsthm,alltt}

\usepackage{graphicx,enumerate}

\usepackage{geometry}
\usepackage{amssymb}
\usepackage{verbatim}
\usepackage{eucal}
\usepackage{mathrsfs}
\usepackage{graphicx}
\usepackage{psfrag}

\newtheorem{theorem}{Theorem}[section]
\newtheorem{lemma}[theorem]{Lemma}
\newtheorem{proposition}[theorem]{Proposition}

\theoremstyle{definition}
\newtheorem{definition}[theorem]{Definition}
\newtheorem{conjecture}[theorem]{Conjecture}
\newtheorem{example}[theorem]{Example}

\theoremstyle{remark}
\newtheorem{remark}[theorem]{Remark}

\begin{document}

\title{Motivic  zeta functions and infinite cyclic covers}
\date{\today}

\author{Manuel Gonz\'alez Villa}
\address{M. Gonz\'alez Villa: Department of Mathematics, University of Wisconsin-Madison, 480 Lincoln Drive, Madison, WI 53706, USA}
\email {villa@math.wisc.edu}

\author{Anatoly Libgober}
\address{A. Libgober: Department of Mathematics, University of Illinois at Chicago, 851 S Morgan Street, Chicago, IL 60607, USA}
\email {libgober@uic.edu}

\author{Lauren\c{t}iu Maxim}
\address{L. Maxim: Department of Mathematics, University of Wisconsin-Madison, 480 Lincoln Drive, Madison, WI 53706, USA}
\email {maxim@math.wisc.edu}

\date{\today}


\dedicatory{To Lawrence Ein on the occasion of his $60$th birthday}


\begin{abstract}
We associate with an infinite cyclic cover of 
a punctured neighborhood of  a simple normal 
crossing divisor on a complex quasi-projective manifold 
(assuming certain finiteness conditions are satisfied)
a rational function  in  $K_0({\rm Var}^{\hat \mu}_{\mathbb{C}})[\LL^{-1}]$, 
which we call {\it motivic infinite cyclic zeta function}, 
and show its birational invariance.
Our construction is a natural extension of the notion of {\it motivic infinite cyclic covers} introduced by the authors, and as such, it generalizes  the Denef-Loeser motivic Milnor zeta function of a complex hypersurface singularity germ. 
\end{abstract}
\maketitle

\section{Introduction}

Motivated by the importance in topology (e.g., in knot theory \cite{Ro}, but see also \cite{MilnorInfCyclic}) and algebraic geometry (e.g., for the study of Alexander-type invariants of complex hypersurface complements, see \cite{DL, DN, L1, L3, MaximComment}) of infinite cyclic covers,  and by the work of Denef and Loeser on motivic invariants of singularities of complex hypersurfaces, in \cite{GVLM} we attached to an infinite cyclic cover associated to a 
punctured neighborhood of a simple normal crossing divisor $E$ on a complex quasi-projective manifold $X$ and a holonomy map $\Delta$, an element $S_{X,E,\Delta}$ in the Grothendieck ring 
$K_0({\rm Var}^{\hat \mu}_{\mathbb{C}})$ of algebraic $\mathbb{C}$-varieties
endowed with a good action of the pro-finite group ${\hat \mu}=\lim \mu_n$ of roots of unity, which we called a {\it motivic infinite cyclic cover}. 

In this note, we extend the above-mentioned construction to attach to an infinite cyclic cover associated to a 
punctured neighborhood of a simple normal crossing divisor $E$ on a complex quasi-projective manifold $X$, together with holonomy $\Delta$ and log discrepancy $\nu$, a rational function $Z_{X, E, \Delta, \nu}(T)$ in $K_0({\rm Var}^{\hat \mu}_{\mathbb{C}})[\LL^{-1}][[T]]$, which we call a {\it motivic infinite cyclic zeta function}. It is worthwhile to notice that, although in \cite{GVLM} we allowed the holonomy map to take arbitrary integer values, here we restrict the holonomy to always take strictly positive values. This technical condition is preserved under blow-ups (hence under birational maps), and it allows us to relate    $Z_{X, E, \Delta, \nu}(T)$  with $S_{X,E,\Delta}$.

Our construction of the motivic infinite cyclic zeta function  is topological in the sense 
that it does not make use of arc spaces. One of our main results, Theorem \ref{Welldefined}, shows that our notion of motivic infinite cyclic zeta function is a birational invariant, or equivalently, $Z_{X, E, \Delta, \nu}(T)$ is an invariant of the punctured neighborhood of $E$ in $X$. The main results of this note are enhanced versions of (and rely on) similar facts about motivic infinite cyclic covers obtained in our earlier paper \cite{GVLM}.

Finally, in the last section we discuss the significance of motivic infinite cyclic zeta functions for the monodromy conjecture.

\smallskip

\textbf{Acknowledgments.}  The authors are grateful to the anonymous referee for carefully reading the manuscript, and for his valuable comments and constructive suggestions. 
 M. Gonz\'alez Villa is partially supported by the grant MTM2013-45710-C2-2-P from Ministerio de Econom\'{\i}a y Competitividad of the Goverment of Spain. A. Libgober is partially supported by a grant from the Simons Foundation. L. Maxim is partially supported by a NSF grant.


\section{Infinite cyclic cover of finite type}\label{s2}

Let $X$ be a smooth complex quasi-projective variety of dimension $r+1$ and $E$ a (reduced) simple normal crossing divisor on $X$ which shall be called a {\it deletion} 
(or {\it deleted) divisor}. 
Let $E = \sum_{i \in J} E_i$ be a decomposition of $E$ into irreducible components $E_i$, where we assume that all divisors $E_i$ are smooth. 
We use the following natural stratification of $X$ given by the intersections of the irreducible components of $E$: for each $I \subseteq J$ consider
\begin{equation}\label{str}E_I = \bigcap_{i \in I}E_i  \quad \hbox{ and } \quad  E^\circ_I = E_I \setminus \bigcup_{j \not \in I} E_j.\end{equation}
Clearly, $X = \bigcup_{I \subseteq J} E^\circ_I$, $X\setminus E= E^\circ_\emptyset$ and $E=\bigcup_{\emptyset \ne I \subseteq J} E^\circ_I$. 

Let $T^*_{X, E}$ be a punctured regular neighborhood of $E$ on $X$, e.g., see \cite[Section 2]{GVLM} and the references therein. Note that $T^*_{X, E}$ is homotopy equivalent to the boundary of a regular neighborhood of $E$ in $X$, which sometimes is called the {\it link} of $E$. 
In the following, we may discard the subscript $X$ and just 
write $T^*_E$. Note that  $T^*_E$ is a union of locally trivial fibrations  $T^*_{E^\circ_I} \rightarrow E^\circ_I$ over the strata $ E^\circ_I$ (with $\emptyset \ne I \subseteq J$) of $E$, the fiber of the latter fibration being diffeomorphic to $(\mathbb{C}^\ast)^{ | I |}$, where $| I |$ denotes the number of elements in the set $I$.

\begin{definition}[\cite{GVLM}]\label{defft} {\it Infinite cyclic cover of finite type.} 
\\ Let $\Delta: \pi_1(T^*_E) \to \ZZ$ be an epimorphism\footnote{The surjectivity assumption is made here solely for convenience (in which case the corresponding infinite cyclic cover is connected), all results in this paper being valid for arbitrary homomorphisms to $\ZZ$.},
 and for any $i \in J$ let $\delta_i$ denote the boundary 
of a small oriented disk  
transversal to the irreducible component $E_i$. 
Let $\widetilde{T}^{*}_{X, E, \Delta}$ be the corresponding infinite cyclic cover (with Galois group $\ZZ$) of the 
punctured regular neighborhood $T^*_E$ of a simple normal crossing
divisor $E \subset X$.
We call such an infinite cyclic cover $\widetilde{T}^{*}_{X, E, \Delta}$ 
{\it of finite type} if $m_i=\Delta(\delta_i)\neq 0$ for all $i \in J$.
Sometimes we omit $\Delta$ and $X$ in the notation 
and write simply $\widetilde{T}^*_E$. 
The map $\Delta$ is also referred to as the {\it holonomy}
of this infinite cyclic cover.
\end{definition}

\begin{remark} Proposition 2.4 of \cite{GVLM} shows that if $\widetilde{T}^*_E$ is an infinite cyclic covers of finite type, then, for any $i \in \ZZ$, the rational vector spaces $H_c^i(\widetilde{T}^*_E)$ and $H^i(\widetilde{T}^*_E)$ are finite dimensional. This fact justifies the terminology used in Definition \ref{defft}. In this note, we deal only with infinite cyclic covers of finite type.\end{remark}

\begin{remark} The infinite cyclic cover $\widetilde{T}^{*}_{X, E, \Delta}$ has the structure of a complex manifold, but it is not an algebraic variety. An algebro-geometric (motivic) realization of $\widetilde{T}^{*}_{X, E, \Delta}$ was given in \cite{GVLM} via the associated motivic infinite cyclic cover $S_{X,E,\Delta}$. Moreover, it was proved in \cite[Theorem 3.7]{GVLM} that $S_{X,E,\Delta}$ is an invariant of the punctured neighborhood of $E$ in $X$, i.e., it is invariant under the following equivalence relation.
\end{remark}

\begin{definition}[\cite{GVLM}]\label{equivnbhd}
 Let $T_1=T_{X_1,E_1}$ and $T_2=T_{X_2,E_2}$ be two regular 
neighborhoods of normal crossing divisors, and 
$\Delta_i: \pi_1(T_i^*) \rightarrow \ZZ$, $i=1,2$, be 
epimorphisms defined on the fundamental groups of the corresponding punctured
neighborhoods. 

We say that the punctured neighborhoods $T^*_1$ and $T^*_2$ of $E_1$ and resp. $E_2$ are {\it equivalent} if there exists a birational map $\Phi:X_1 \to X_2$, which is biregular on $T^\ast_1 \subset X_1$ (and, respectively, $\Phi^{-1}$ is biregular on $T^\ast_2 \subset X_2$) and which moreover induces a map $\Phi\vert_{T_1^*}: T_1^* \rightarrow T_2^*$ such that $\Phi(T_1^*)$ and $T_2^*$ are deformation retracts of 
each other.

We say that the pairs $(T_1^*,\Delta_1)$ and $(T_2^*,\Delta_2)$ 
are {\it equivalent} if the punctured neighborhoods $T^*_1$ and $T^*_2$ are {equivalent}  in the above sense, 
 and  the following diagram commutes:
\begin{equation}\label{commutativity}
\xymatrix{
\pi_1(T_1^*)\ar[rd]_{\Delta_1}  \ar[rr]^{(\Phi\vert_{T^*_1})_*} && \pi_1(T_2^*) \ar[ld]^{\Delta_2}\\
& \ZZ &
}
\end{equation}
Here $(\Phi\vert_{T^*_1})_*$ is the homomorphism induced by  $\Phi\vert_{T_1^*}$ on the 
fundamental groups.
\end{definition}

The basic example of equivalent pairs $(T_1^*,\Delta_1)$ and $(T_2^*,\Delta_2)$, which is crucial for the proofs of Theorem \ref{Welldefined} and \cite[Theorem 3.7]{GVLM} is the following. 

\begin{example}\label{blowupholomony}
Let $(X,E)$ be as above, and consider the blow-up $p : X' := Bl_Z X \rightarrow X$  of $X$ along the smooth center $Z \subset X$ of codimension $\geq 2$ in $X$. Assume moreover that  $Z$ is contained in $E$ and  has normal crossings with the components of $E$. Denote by $E_*$  the exceptional divisor of the blow-up $p$, which is isomorphic to the projectivized normal bundle over $Z$, i.e.,  $E_* \cong \mathbb{P}(\nu_Z)$. 
 Let us denote the preimage  of the divisor $E_i$ in $X'$ by $E'_i$. Denote by $E'$ the normal crossing divisor in $X'$ formed by the $E'_i$ together with $E_\ast$. 
Assume that there is $I \subseteq J$ maximal such that $Z$ is contained in $E_I$. We 
consider the (surjective) homomorphism given by  
the composition $$\Delta' : \pi_1({T}^*_{X', E'}) 
\rightarrow \pi_1(T^*_{X,E}) 
\buildrel \Delta \over
\rightarrow \mathbb{Z},$$ 
resulting from the identification $T^*_{X'E'} \overset{\cong}{\rightarrow} T^*_{X,E}$ 
induced by the blow-down map and the fact that the center $Z$ is contained in $E$. 
We then have: 
\begin{equation}\label{blowupholomonyeqn}\Delta'(\delta'_i)= \Delta(\delta_i)=m_i \quad  (i \in J) \ \ \hbox{ and } \ \ m_*:=\Delta'(\delta_*) = \sum_{i \in I} m_i,\end{equation} where $\delta'_i$ and $\delta_*$ are the meridians about the components $E'_i$ and $E_*$ of $E'$.  Indeed, the blow-down map 
takes the $2$-disk transversal to $E_*$ (at a generic point) and 
bounded by $\delta_*$, to 
the disk in $X$ transversal to the components $E_i, i\in I$, containing 
$Z$ and disjoint from the remaining components of $E$, i.e., 
one has the relation $\delta_*=\sum_{i \in I} \delta_i$ in $H_1(T^*_E)$.\end{example}

  
\section{Motivic infinite cyclic zeta functions}\label{s3}
Most of our calculations will be done in the ring of power series with coefficients in the localized  Grothendieck ring $\mathcal{M}^{\hat{\mu}}_{\CC}=K_0({\rm Var}^{\hat{\mu}}_{\CC})[\LL^{-1}]$, where ${\rm Var}^{\hat{\mu}}_{\CC}$ is the category of complex algebraic varieties  endowed with good $\hat{\mu}$-actions. Let us recall some definitions, e.g., see \cite{Barcelona}.

For a positive integer $n$, we denote by $\mu_n$ the group of all $n$-th roots of unity (in $\mathbb{C}$).  The groups $\mu_n$ form a projective system with respect to the maps   $\mu_{d \cdot n} \to \mu_n$ defined by $ \alpha \mapsto \alpha^d$, and we denote by $\hat \mu:=\lim \mu_n$ the projective limit of the $\mu_n$.

Let $X$ be a complex algebraic variety. A good $\mu_n$-action on $X$ is an algebraic action $\mu_n \times X \to X$, such that each orbit is contained in an affine subvariety of $X$. (This last condition is automatically satisfied if $X$ is quasi-projective.) A {good} $\hat\mu$-action on $X$ is a $\hat\mu$-action which factors through a good $\mu_n$-action, for some $n$.

The Grothendieck ring $K_0({\rm Var}^{\hat{\mu}}_{\CC})$ of  the category ${\rm Var}^{\hat{\mu}}_{\CC}$ of complex algebraic varieties  endowed with a good $\hat{\mu}$-action is generated by classes $[Y, \sigma]$ of isomorphic varieties endowed with good $\hat\mu$-actions, modulo  the following relations:
\begin{itemize}
\item $[Y, \sigma] = [Y \setminus Y', \sigma_{|_{Y \setminus Y'}}] + [Y', \sigma_{|_{Y'}}]$,  if $Y'\subseteq Y$ is a closed $\sigma$-invariant subset.
\item $[Y \times Y', (\sigma, \sigma')] = [Y, \sigma][Y', \sigma']$.
\item $[Y \times \mathbb{A}^1_\mathbb{C} , \sigma] = [Y \times \mathbb{A}^1_\mathbb{C}, \sigma']$,  if $\sigma$ and $\sigma'$ are two affine liftings of the $\mathbb{C}^\ast$-action on $Y$.
\end{itemize} 
The third relation above is included for completeness. However, it is not needed in this paper.

We let $\mathbb{L}$ be the class in $K_0({\rm Var}^{\hat{\mu}}_{\CC})$ of $\mathbb{A}^1_\mathbb{C}$, with the trivial $\hat\mu$-action. We denote by $\mathcal{M}^{\hat{\mu}}_\CC$ the ring obtained from $K_0({\rm Var}^{\hat{\mu}}_{\CC})$ by inverting $\LL$.

\medskip

We now introduce the main objects of the paper. Firstly, recall that $J$ denotes the index set for the irreducible components of the deletion divisor $E$. For $I \subseteq J$, we can use the exact sequence 
$$\mathbb{Z}^{| I|} \rightarrow \pi_1(T^\ast_{E^\circ_I}) \rightarrow \pi_1(E^\circ_I) \rightarrow 0$$ of the locally trivial fibration  $T^\ast_{E^\circ_I} \rightarrow E_I^{\circ}$ together with the (restricted) holomony $\Delta: \pi_1(T^\ast_{E^\circ_I})\rightarrow \ZZ$ to define a map 
 \begin{equation}\label{maptildeEI}
\Delta_I : \pi_1(E^\circ_I) \rightarrow \ZZ/ m_I\ZZ \end{equation}
where $m_I : = \gcd(m_i | i\in I)$ is the index of the image of $\ZZ^{\vert I \vert}$ in $\ZZ$. For complete details,  see Lemma 3.1 and 3.2 in \cite{GVLM}. 

\begin{definition}\label{ucov} We denote by  $\widetilde{E^\circ_I}$ the unramified cover of  $E^\circ_I$ with Galois group $\ZZ / m_I\ZZ$ induced by the map $\Delta_I$ from (\ref{maptildeEI}).\end{definition}
  
Remark that the cover $\widetilde{E^\circ_I}$ is an algebraic variety endowed with a good $\mu_{m_I}$-action. The cover $\widetilde{E^\circ_I}$  has $n$ connected components, where $n$ is the index of $\Delta(\pi_1(T^\ast_{E^\circ_I}))$ in $\ZZ$.  Each of the connected components is the cyclic cover
of $E_I^\circ$ with the covering group $n\ZZ/ m_I\ZZ$. 

\begin{remark}\label{fib}
As shown in \cite{GVLM}, the locally trivial $(\CC^{\ast})^{|I|}$-fibration $T^*_{E^\circ_I} \rightarrow E^\circ_I$ induces a fibration
$$\widetilde{T}^*_{E^\circ_I} \rightarrow \widetilde{E}^\circ_I$$ on the infinite cyclic cover $\widetilde{T}^*_{E^\circ_I}$, 
with connected fiber $\widetilde{(\CC^{\ast})^{|I|}}\cong (\CC^{\ast})^{|I|-1}$, the infinite cyclic cover of $(\CC^{\ast})^{|I|}$ defined by $\ker(\ZZ^{|I|} \to m_I\ZZ)$.
\end{remark}

In order to define the motivic zeta function of an infinite cyclic cover,  we also need the notion  of {\it log discrepancy} with respect to a pair in the sense of, for instance, Kollar and Mori \cite{KollarMori}. The actual definition of a log discrepancy is not needed here, but the reader may consult \cite[Section 2.2 and Definition 2.25]{KollarMori} for complete details. 
Let us only recall that for a pair $(X, E)$ as in our setup, a log discrepancy $\nu$ with respect to $(X,E)$ is determined by a collection of integers $\{\nu_i\}_{i \in J}$, one for each irreducible component $E_i$ of $E$, $i \in J$, and moreover it satisfies the following blow-up relation (\ref{eleprocld}), which can be derived from  \cite[Lemma 2.29]{KollarMori} under the additional hypothesis listed below. 

Let $X'$ be the blow-up $p: X' := Bl_ZX \rightarrow X$ of $X$ along a smooth center $Z \subset X$ of codimension $\geq 2$.  Assume moreover  that the center $Z$ is contained  in $E$ and  that it has normal crossings with the components of $E$, as in Example \ref{blowupholomony}.  Denote by  $E_\ast$ the exceptional divisor of the blow-up, and by $E'_i$ the proper transform of $E_i$ ($i \in J$), with $E'=E_*+\sum_{i\in J}E'_i$. Then the log discrepancies $\nu'_i$ (resp. $\nu_\ast$) of $E'_i$   (resp. $E_\ast$) satisfy
\begin{equation}\label{eleprocld}
\nu'_i = \nu_i \ \hbox{ for all } i \in J, \  \hbox{ and } \ \nu_\ast = \sum_{i \in I} \nu_i + c, 
\end{equation}
where $I$ is the maximal subset of $J$ such that $Z$ is contained in $E_I$, and $c$ denotes the codimension of $Z$ in $E_I$. 

\begin{example}\label{diffformdiscr}
(i) A global holomorphic form $\omega$ on $X$ supported along a simple normal crossing divisor $E$ defines a log-discrepancy $\nu$, with $\nu_i$ given by the multiplicity  of $E_i$ in the divisor of the form $\omega$ plus one, i.e., $\nu_i -1 = {\rm ord}_{E_i} \omega$.

(ii)  Let $\{ f=0 \}$ be a germ of holomorphic function at the
  origin in $\CC^n$, let  $\pi: X \rightarrow \CC^n$ be an embedded resolution of the singularities of $f$, and let $E$ be the total transform of $\{f=0\}$.  The standard volume form $\omega = dz_1\wedge...\wedge dz_n$ gives a log-discrepancy of the pair $(X, E)$, see \cite{I}. See also Section \ref{relationwithmotivic}.\end{example}

\begin{remark} A blow-up $p: X' := Bl_ZX \rightarrow X$ as above identifies the set of irreducible components, labelled by the set of indices $J$, 
of the normal crossing divisor $E$ on $X$ with a subset of irreducible components of the divisor $E'$ in $X'$. Moreover, the punctured neighborhood of $E$ in $X$ and that of $E'$ in  $X'$ are homeomorphic. 
For birational maps $X
\rightarrow Y,$ i.e., for compositions of blow-ups and blow-downs, there are similar identifications of certain components of normal crossing divisors in $X$ and resp. $Y$ which have homeomorphic punctured neighborhoods.
The blow-up condition (\ref{eleprocld}) may impose in this case additional relations on the values $\nu_i$ of a log-discrepancy with respect to a pair $(X,E)$. 
For example, suppose that the exceptional divisor in a $4$-manifold $X'$
is $E_*=\PP^2 \times \PP^1$ and can be contracted to either
$Z_1=\PP^2$ in $X_1$ or $Z_2=\PP^1$ in $X_2$. Suppose a divisor $F$ intersects $E_*$
in $\PP^2 \times P$, with $P \in \PP^1$, and a divisor $G$ intersects $E_*$ along $\PP^1 \times \PP^1$.
Then, since the codimension of $Z_1$ in $F$ is $1$ and the codimension of $Z_2$ in $G$ is $2$, and the corresponding values $\nu_F$ and $\nu_G$ of the log-discrepancies are invariant because of the first relation in (\ref{eleprocld}), 
it follows that  $\nu_\ast = \nu_F+1=\nu_G+2$.
\end{remark}

The previous remark motivates the following definition.

\begin{definition}\label{equivnbhdld} Let $T_1=T_{X_1,E_1}$ and $T_2=T_{X_2,E_2}$ be regular 
neighborhoods of normal crossing divisors as above, with holonomies  
$\Delta_i: \pi_1(T_i^*) \rightarrow \ZZ$, $i=1,2$, and log discrepancies $\nu^i$, $i=1,2$ defined with respect to  $(X_1,E_1)$ (resp. $(X_2,E_2)$). We say that the triples $(T_1^*,\Delta_1, \nu^1)$ and $(T_2^*,\Delta_2, \nu^2)$ 
are {\it equivalent} if $(T_1^*,\Delta_1)$ and $(T_2^*,\Delta_2)$ are equivalent in the sense of Definition \ref{equivnbhd} and  the log-discrepancies $\nu^1$ and $\nu^2$ are connected by a finite sequence of blow-up relations as in (\ref{eleprocld}).\end{definition}

The following Definition and Lemma provide a natural situation where log-discrepancies exist.

\begin{definition} Let $T_E^*$ be a punctured neighborhood of a normal 
crossing divisor $E$ on a smooth projective manifold $X$. A holomorphic volume form
on $T_E^*$ is a meromorphic form $\omega$ on $X$ such that the irreducible
components of the divisor associated to $\omega$ are either irreducible components of $E$ or
have empty intersection with $E$. 
\end{definition}

\begin{lemma} Let $T^*_E$ be a punctured neighborhood of a normal 
crossing divisor $E$ on a smooth projective manifold $X$ and let $\omega$ be a holomorphic volume form on $T^*_E$. Then
$\nu(E_i)={\rm ord}_{E_i}(\omega)-1$ defines a log-discrepancy, i.e.,  
satisfies  the blow-up relations  in (\ref{eleprocld}). 
\end{lemma}

\begin{example} Let $\{f=0\}$ be a germ of holomorphic function at the
  orgin in $\CC^n$ as in Example \ref{diffformdiscr} (ii).  The standard volume form $\omega = dz_1\wedge...\wedge dz_n$ defines a  holomorphic volume form on a tubular neighborhood of the strict transform of $\{ f=0\}$ in the compactification of $X$, where $\pi : X \rightarrow \CC^n$ is an embedded resolution of the singularities of $f$. The corresponding log-discrepancies coincide with those used by Igusa in \cite{I}.
\end{example}

We introduce now the main object of this paper, the motivic infinite
cyclic zeta function. In order to do 
so we will assume from now that the holonomy map $\Delta$ takes only strictly positive values. The reason for this assumption is two-fold. This condition is enough for our  definition and it is stable under blow-ups, see Remarks \ref{rmk1} and \ref{rafterbi}.  It also allows to relate motivic infinite cyclic zeta functions and the motivic infinite cyclic covers from \cite{GVLM}, see Remark \ref{lznes}.

\begin{definition}\label{maindef} {\it Motivic zeta function associated to an infinite cyclic cover of finite type and a choice of log-discrepancy}
\newline Let $T_E^*$ be the punctured neighborhood  
of a normal crossing divisor $E$ in 
a quasi-projective manifold $X$ as in Section \ref{s2}, let $\Delta: \pi_1(T^*_E) \rightarrow \ZZ$ be an epimorphism such that 
the corresponding infinite cyclic cover $\widetilde{T}_{X,E}^*$ 
is of finite type, and let $\nu$ be a log-discrepancy with respect to $(X, E)$.  Assume there is a choice of orientations of all $\delta_i$ with $i \in J$ such that $m_i > 0$. 
For each fixed subset $A \subseteq J$, we define the corresponding 
{\it motivic infinite cyclic zeta function} (of finite type) of $T^*_E$ as
\begin{equation}\label{defnmzf}Z^A_{X, E, \Delta, \nu}(T) :=  \sum_{\substack{\emptyset \not = I \subseteq J\\ A \cap I \ne \emptyset}} [\widetilde{E^\circ_I}, \sigma_I](\mathbb{L}-1)^{|I| -1} \prod_{i \in I} \frac{\mathbb{L}^{-\nu_i}T^{m_i}}{1 - \mathbb{L}^{-\nu_i}T^{m_i}} \in \mathcal{M}^{\hat{\mu}}_{\CC}[[T]]_{{\rm sr}},
\end{equation}
where $\mathcal{M}^{\hat{\mu}}_{\CC}[[T]]_{{\rm sr}}$ is the $\mathcal{M}^{\hat{\mu}}_{\CC}$-submodule of \emph{rational series} of $\mathcal{M}^{\hat{\mu}}_{\CC}[[T]]$ generated by $1$ and finite products of terms $\frac{\mathbb{L}^{-\nu_i}T^{m_i}}{1 - \mathbb{L}^{-\nu_i}T^{m_i}}$ with $(m_i, \nu_i) \in \ZZ_{>0} \times \ZZ$.

Sometimes we omit  $\Delta$ and $\nu$ from the notation. In the case $A=J$, we use the simpler notation $Z_{X, E, \Delta, \nu}(T)$ or $Z_{X, E}(T)$. 

In what follows, we will also use the notation $$A_i:= \mathbb{L}^{-\nu_i}T^{m_i}$$ for any $i \in J$. 
\end{definition}

\begin{remark}\label{rmk1}The factors $A_i(1-A_i)^{-1}$ in (\ref{defnmzf}) are well-defined since $m_i > 0$ for all $i$.\end{remark}

Before stating and proving the main result of this section, let us discuss the relation between the motivic zeta function associated to an infinite cyclic cover of finite type  and the motivic infinite cyclic cover introduced in \cite{GVLM}. First, let us recall the definition of the motivic infinite cyclic cover: 

\begin{definition}[\cite{GVLM}]\label{micc} For each fixed subset $A \subseteq J$, we define the corresponding 
{\it motivic infinite cyclic cover} (of finite type) of $T^*_E$ by
\begin{equation}\label{defnmotinfcover}S^A_{X, E, \Delta} :=  \sum_{\substack{\emptyset \not = I \subseteq J\\ A \cap I \ne \emptyset}} (-1)^{|I|-1}[\widetilde{E^\circ_I}, \sigma_I](\mathbb{L}-1)^{|I| -1} \in K_0({\rm Var}^{\hat{\mu}}_{\CC}).
\end{equation}
As before, when $A=J$, we use the notation $S_{X, E, \Delta}$ or $S_{X, E}$.
\end{definition}

\begin{remark}By constrast with Definition \ref{maindef}, the definition of motivic infinite cyclic cover  does not make any extra assumption  on the holonomy map.  The reason we make such restrictions in Definition \ref{maindef} is to be able to recover the motivic infinite cyclic cover $S^A_{X, E, \Delta}$ from $Z^A_{X,E,\Delta, \nu}(T)$.\end{remark}

There exists a unique $\mathcal{M}^{\hat{\mu}}_\mathbb{C}$-linear morphism
\begin{equation}\label{deflim}\lim_{T \to + \infty} : \mathcal{M}^{\hat{\mu}}_\mathbb{C}[[T]]_{{\rm sr}} \rightarrow \mathcal{M}^{\hat{\mu}}_\mathbb{C}\end{equation}
such that \begin{equation}\lim_{T \to + \infty} \prod_{i \in I} \frac{\mathbb{L}^{-\nu_i}T^{m_i}}{1 - \mathbb{L}^{-\nu_i}T^{m_i}} = (-1)^{|I|},
 \end{equation} when $(m_i, \nu_i) \in \mathbb{Z}_{>0} \times \mathbb{Z}$ (see \cite[Section 2.8]{GLM} or  \cite[Lemma 4.1.1]{DenefLoeserJAG}). This morphism allows to relate the infinite cyclic zeta function and the motivic infinite cyclic cover as follows 
\begin{equation}\label{lzs}{S}^A_{X,E,\Delta} =  - \lim_{T \to + \infty} Z^A_{X, E, \Delta, \nu}(T) \in \mathcal{M}^{\hat{\mu}}_\mathbb{C}.\end{equation}
 
\begin{remark}\label{lznes} Note that the motivic infinite cyclic cover forgets all information about the log-discrepancy $\nu$. However, equation (\ref{lzs}) does not hold if one allows the holomony map to take non-positive values. Indeed, if $m_i<0$ then we have that $\lim_{T \to + \infty}  A_i (1 - A_i)^{-1} =0$, and if $m_i=0$ then $\lim_{T \to + \infty}  A_i (1 - A_i)^{-1} =A_i (1 - A_i)^{-1}$. \end{remark}

The main result of this section is the following. 
\begin{theorem}\label{Welldefined} The above notion of motivic infinite cyclic zeta function is an invariant of $(T^*_{X,E},\Delta, \nu)$, i.e., it is preserved under the equivalence relation of Definition \ref{equivnbhdld}.
\end{theorem}

The proof of Theorem \ref{Welldefined} is similar to (and relies on) that of \cite[Theorem 3.7]{GVLM}, and we will use freely facts proved in loc. cit. In fact, since any birational map $X_1 \rightarrow X_2$ providing an equivalence
between punctured neighborhoods (cf. Definition \ref{equivnbhd}) is, 
by the Weak Factorization Theorem \cite{AKMW} (see also \cite{Bon} for the non-complete case), 
a composition of blow-ups and blow-downs, each inducing an equivalence 
between the corresponding punctured neighborhoods,
{\it it suffices to show that the above expression (\ref{defnmzf}) is invariant under blowing up along a smooth center in $E$}. Let us consider $$p : X' := Bl_Z X \rightarrow X$$ the blow-up  of $X$ along the smooth center $Z \subset E$ of codimension $\geq 2$ in $X$. Denote by $E_*$  the exceptional divisor of the blow-up $p$, which is isomorphic to the projectivized normal bundle over $Z$, i.e.,  $E_* \cong \mathbb{P}(\nu_Z)$. We may also assume that the center $Z$ of the blow-up  
is contained in $E$ and 
has normal crossings with the components of $E$ (cf. \cite[Theorem 0.3.1,(6)]{AKMW}).  Let us denote as before the preimage  of the divisor $E_i$ in $X'$ by $E'_i$. Denote by $E'$ the normal crossing divisor in $X'$ formed by the $E'_i$ together with $E_\ast$. Let $J' = J \cup \{ \ast \}$ be the family of indices of the divisor $E'$. For $I \subseteq J$ we denote by $I' \subseteq J'$  the family $I \cup \{ \ast \}$.  Finally, let $A' = A \cup \{\ast\}$.

By the above reduction to the normal crossing situation, we may assume that there is $I \subseteq J$ maximal such that $Z$ is contained in $E_I$. We 
consider the (surjective) homomorphism given by  
the composition $$\Delta' : \pi_1({T}^*_{X', E'}) 
\rightarrow \pi_1(T^*_{X,E}) 
\buildrel \Delta \over
\rightarrow \mathbb{Z},$$ 
resulting from the identification $T^*_{X'E'} \overset{\cong}{\rightarrow} T^*_{X,E}$ 
induced by the blowing down map, as in Example \ref{blowupholomony}. We also consider the log discrepancy $\nu'$ with respect to  $(X',E')$ obtained from $\nu$ by the blow-up relation  
(\ref{eleprocld}). Note that $\widetilde{T}^{*}_{X', E',\Delta'}$ is of finite type since $\widetilde{T}^*_{X,E,\Delta}$ is so and $T^*_{X',E'} \cong T^*_{X,E}$. Moreover, by Lemma 3.2 in \cite{GVLM}  and Definition \ref{maindef} applied to $(X',E', \Delta',\nu')$, we can define the corresponding motivic infinite cyclic zeta function by:
\begin{equation}\label{afterbl}Z^{A'}_{X', E', \Delta', \nu'}(T) := \sum_{\substack{\emptyset \not = K \subset J' \\ K \cap A' \ne \emptyset}} [\widetilde{E^\circ_K}, \sigma_{K}](\mathbb{L}-1)^{|K| -1}\prod_{k\in K} \frac{\LL^{-\nu_k}T^{m_k}}{1-\LL^{-\nu_k}T^{m_k}}.\end{equation}
We will also sometimes denote $\LL^{-\nu_\ast}T^{m_\ast}$ by $A_\ast$. 

\begin{remark}\label{rafterbi} The factor $A_\ast (1 - A_\ast)^{-1}$ in (\ref{afterbl}) is well-defined because $m_i > 0$ for all $i$, and therefore $m_\ast = \sum_{i \in I} m_i >0$. \end{remark}

\medskip

 Theorem \ref{Welldefined} follows now from the following proposition. 
\begin{proposition}\label{ind} With the above notations, we have the following identification:
\begin{equation}\label{S=S}Z^A_{X, E,\Delta, \nu}(T)  = Z^{A'}_{X', E',\Delta', \nu'}(T) \in \mathcal{M}^{\hat{\mu}}_\mathbb{C}[[T]]_{{\rm sr}}.\end{equation}
\end{proposition}

\bigskip

Note that we can always restrict the comparison of  zeta functions in Proposition \ref{ind} to strata in the center of blowup and in the exceptional divisor, respectively. Indeed, the blow-up map induces an isomorphism outside the center $Z$, so the strata in $E\setminus Z$ and $E' \setminus E_*$ are in one-to-one isomorphic correspondence; moreover, these isomorphisms can be lifted (e.g., by Lemma 3.2 in \cite{GVLM}) to the corresponding unramified covers. It also suffices to prove the above result only in the case $A=J$.

The proof of Proposition \ref{ind} is by induction on the dimension of the center of blow-up.


\subsection{Beginning of induction}

Let us consider the following examples  in relation to the starting case of induction, i.e., when the center $Z$ is a point.

\begin{example}\label{A}
Let  $X$ be a surface and let $E_1$ and $E_2$ be two smooth curves  intersecting transversally at a point $P$. Let us consider the blow-up 
$X'=Bl_Z X$ of $X$ at the center $Z=P$. The exceptional divisor is $E_* \cong \PP^1$ and we have  $E_*^\circ \cong \mathbb{C}^*$. 
Let $\delta_{i} \in H_1(T_{E_i^{\circ}}^*,\ZZ)$ ($i=1,2$) be the class of the fiber of the
projection of punctured neighborhood $T_{E_i^{\circ}}^*$ onto the stratum $E_i^{\circ}$.
If $\Delta(\delta_{1})= m_1$,  $\Delta(\delta_{2})=m_2$, $\nu(\delta_{1})=\nu_1$, $\nu(\delta_2)=\nu_2$,  and we let $m= {\rm gcd}(m_1, m_2)$,  then the contribution of $P$ to $Z_{X, E}(T)$ is 
$$[\mu_m](\mathbb{L}-1)\frac{A_1}{1- A_1}\frac{A_2}{1 - A_2},$$ 
and the contributions of the exceptional divisor $E_*$ to $Z_{X', E'}(T)$ are
$$\frac{A_\ast}{1 - A_\ast} \left ( (\LL-1) \left ( [\widetilde{E'_1 \cap E_*}]\frac{A_1}{1 - A_1}  + [\widetilde{E'_2\cap E_*}]\frac{A_2}{1 - A_2} \right ) + [\widetilde{E_*^\circ}] \right ).$$ 
Because $m = {\rm gcd}(m_1, m_1 + m_2) = {\rm gcd}(m_2, m_1 + m_2)$, we get the following equalities:  $[\widetilde{E'_1 \cap E_*}, \sigma_{\Delta'}] = [\widetilde{E'_2\cap E_*}, \sigma_{\Delta'}] = [\mu_m]$. Finally, it follows from Lemma 3.2 in \cite{GVLM}  that $[\widetilde{E_*^\circ}, \sigma_{\Delta'}]=[\mu_m](\mathbb{L}-1)$, see  \cite[Example 3.9]{GVLM} for more details. Hence, after factoring out  $[\mu_m](\LL-1)$, it remains to show that the two 
 contributions 
 $$\frac{A_1}{1- A_1}\frac{A_2}{1 - A_2}
\quad {\rm  and } \quad 
 \frac{A_\ast}{1 - A_\ast} \left [ \left (\frac{A_1}{1 - A_1}  + \frac{A_2}{1 - A_2} \right ) + 1 \right ]$$ 
 coincide.
After reducing to the common denominator $(1-A_1)(1-A_2)(1 - A_\ast)$, we just need to check that 
 $$A_1 A_2 ( 1 - A_\ast) =  A_1 ( 1 - A_2) A_\ast +  ( 1 - A_1) A_2 A_\ast + A_\ast (1-A_1)(1-A_2),$$
 which further simplifies to $A_1A_2= A_\ast$. The latter claim holds since in this case $m_\ast = m_1 + m_2$ and $\nu_\ast = \nu_1 + \nu_2$ (here $\codim_{E_I}Z=0$). 

Note that in the case when $P$ belongs to only one irreducible component, say $E_1$,  we have  $[\widetilde{E}^\circ_1 |_P]= [\mu_{m_1}]$   and $[\widetilde{E_*^\circ}, \sigma_\Delta] = [\mu_{m_1}] \mathbb{L}$. In this case, the contributions of $P$ and resp. $E_*$ to $Z_{X, E}(T)$  and resp. $Z_{X', E'}(T)$ are $$[\mu_{m_1}]\frac{A_1}{1 - A_1} \quad {\rm and \ resp.} \quad [\mu_{m_1}](\mathbb{L} -1)\frac{A_1}{1 - A_1} \frac{A_\ast}{1 -A_\ast} + [\mu_{m_1}] \mathbb{L} \frac{A_\ast}{1 - A_\ast}.$$
Again, after factoring out $[\mu_{m_1}]$ and reducing to common denominator, we need to check that  
$$A_1 (1 - A_\ast) = (\LL-1)A_1 A_\ast + \LL (1 - A_1) A_\ast,$$
which simplifies to $A_1= \LL A_\ast$. The latter claim holds since  $m_\ast=m_1$ and $\nu_\ast= \nu_1 + 1$ (here $\codim_{E_I}Z=1$). \hfill$\square$\end{example}

\begin{example}\label{B} Let $X$ be a threefold and $E_1$, $E_2$, $E_3$ be three divisors 
intersecting transversally at a point $P$. Consider the blow-up $X'$ of 
$X$ at the center $Z=P$, so here $I=J=\{1,2,3\}$. The divisor $E=\sum_{i=1}^3 E_i$ of $X$ transforms into the divisor 
$E'$ in $X'$ consisting of  the proper transforms $E'_i$ of the irreducible 
components $E_i$ of  $E$ ($i=1,2,3$),  
together with the exceptional component $E_*\cong \PP^2$. 
As already mentioned, it suffices to restrict the comparison of  $Z_{X,E}(T)$ and $Z_{X',E'}(T)$ only to contributions coming from the strata in the center of blow-up and the exceptional divisor, respectively. 
Let $m_I:=\gcd(m_1,m_2,m_3)$. Then 
 the contribution of $P$ to  $Z_{X,E}(T)$ is 
$$[\mu_{m_I}](\LL - 1)^2 \frac{A_1}{1 - A_1} \frac{A_2}{1 -A_2}\frac{A_3}{1 - A_3}.$$
The exceptional divisor $E_*$ acquires seven strata induced from the stratification of $E'$. These strata are: 
\begin{itemize}
\item $L_{\{i,j\}}=E_* \cap E'_i \cap E'_j$, for $i,j \in \{1,2,3\}$ with $i \neq j$, 
\item  $L_{\{i\}}=(E_* \cap E'_i) \setminus (L_{\{i,j\}} \cup L_{\{i,k\}})$, with $\{i,j,k\}=\{1,2,3\}$,  
\item $E_*^{\circ}=E_*\setminus \bigcup_{i=1}^3 E'_i$.
\end{itemize}
Note that  the strata $E_*^{\circ}$ and $L_{\{i\}}$ are complex tori of dimension 
$2$ and $1$, respectively, while the strata $L_{\{i,j\}}$ are points.

It follows from Lemma 3.2 of \cite{GVLM} 
that
for each of seven strata of $E_*$, the corresponding unbranched covers appearing in (\ref{defnmzf})  
 have $m_I=\gcd(m_1,m_2,m_3)$ components, each of which is biregular to the stratum itself (since all these strata are tori), see Example 3.6 in \cite{GVLM} for more details. 
Hence the contribution of $E_*$ to $Z_{X',E'}(T)$ is:
\begin{equation*}\begin{split}
[\mu_{m_I}] & (\mathbb{L}-1)^2 \frac{A_\ast}{1 - A_\ast} \left ( 1 + \frac{A_1}{1 - A_1} + \frac{A_2}{1 - A_2} + \frac{A_3}{1 - A_3}  \right . \\
& \left . + \frac{A_1}{1 - A_1} \frac{A_2}{1 - A_2}+ \frac{A_1}{1 - A_1} \frac{A_3}{1 - A_3} + \frac{A_2}{1 - A_2} \frac{A_3}{1 - A_3} \right).
\end{split}
\end{equation*}
After we factor out $[\mu_{m_I}](\mathbb{L}-1)^2$ and  reduce to common denominator, we need to check that $A_1 A_2 A_3 (1 - A_\ast)$ is equal to 
\begin{equation*}\begin{split}
A_\ast & \Big [ (1 - A_1) (1 - A_2) (1 - A_3) +  A_1 (1 - A_2) (1 - A_3) +  (1 - A_1) A_2 (1 - A_3) \\ 
&+  (1 - A_1) (1 - A_2) A_3 + A_1 A_2 (1 - A_3) +  A_1 (1-A_2) A_3 +  (1 - A_1) A_2 A_3 \Big ],
\end{split}
\end{equation*}
which further reduces to showing that $A_1 A_2 A_3 = A_\ast$. This last equality holds because $m_\ast = m_1 + m_2 + m_3$ and $\nu_\ast = \nu_1 + \nu_2 + \nu_3$ (here $\codim_{E_I}Z=0$). \hfill$\square$
\end{example}

\begin{example}\label{C} Let $X$ be a threefold, and $E=E_1+E_2$ be a simple normal crossing divisor on $X$, with holonomy values $m_1$ and resp. $m_2$ on the meridians about its irreducible components. Let $m=\gcd(m_1,m_2)$. Choose a point $Z$ contained in the (one-dimensional) intersection $E_J:=E_1 \cap E_2$, for 
 $J=I=\{1,2\}$, and consider the blow-up $X'=Bl_Z X$ of $X$ along the center $Z$.  We denote the exceptional divisor $\PP(\nu_Z)$ by $E_*$.  The divisor $E$ is transformed under the blow-up into the divisor $E'$ in $X'$ consisting of the proper transforms $E'_i$ ($i \in J$) of the irreducible components $E_i$ of $E$, together with the exceptional divisor $E_*\cong \PP^2$.
 
Let us explicitly describe the contribution of the center $Z$ and that of the exceptional divisor $E_*$ to the zeta functions $Z_{X,E}(T)$ and $Z_{X',E'}(T)$, respectively. 
 Clearly, the class $[\widetilde{E_J}|_Z,\sigma_{\Delta}]$ equals $[\mu_m]$. So the contribution to $Z_{X,E}(T)$ consists of 
 $$[\mu_m](\LL-1)\frac{A_1}{1 - A_1} \frac{A_2}{1 - A_2}.$$
 On the other hand, the exceptional divisor $E_*$ acquires four strata induced from the stratification of $E'$, namely,
 \begin{itemize}
 \item $L_J=E'_1 \cap E'_2 \cap E_*$, which is just a point.
 \item $L_{\{i\}}=E_* \cap E'_i \setminus L_J \cong \CC$, for $i \in J$.
 \item $E^{\circ}_*=E_* \setminus (E_1' \cup E_2') \cong \CC \times \CC^*$.
 \end{itemize}
 Note that, since any of the four strata in $E_*$ is either simply-connected or a product of a simply-connected space with a torus, any finite connected unbranched cover of such a stratum is biregular to the stratum itself. Moreover, as shown in \cite[Example 3.11]{GVLM}, for each of the four strata of $E_*$, the corresponding unbranched cover appearing in (\ref{defnmzf}) has exactly $m$ connected components. 
 So the contribution of $E_*$ to  $Z_{X',E'}(T)$ is given by:
\begin{equation}\label{pos}[\mu_m](\LL-1)\frac{A_\ast}{1 - A_\ast} 
  \left[ \LL  + \LL  \left ( \frac{A_1}{1 - A_1} +  \frac{A_2}{1 - A_2}  \right) + (\LL-1) \frac{A_1}{1 - A_1} \frac{A_2}{1 - A_2}  \right].\end{equation}
After factoring out $[\mu_m](\LL-1)$ out and reducing to common denominator, it remains to check the equality
\begin{multline*}
A_1 A_2 (1 - A_\ast) = A_\ast \left[ \LL  (1 - A_1) (1 - A_2) +  \LL  \big(A_1 (1 - A_2) +  (1 - A_1) A_2 \big) +  (\LL - 1)   A_1 A_2 \right],\end{multline*}
which simplifies to $A_1 A_2 = \LL A_\ast$.  The latter equality holds because $m_\ast = m_1 + m_2$ and $\nu_\ast = \nu_1 + \nu_2 + 1$ (here $\codim_{E_I}Z=1$). \hfill$\square$
\end{example}

\begin{example} Let $X$ be a threefold, and $E= E_1$ a (simple normal crossing) divisor on $X$, with holonomy value $m:=m_1$.
Choose $Z$ a point contained in $E$ and consider the blow-up $X' = Bl_Z X$ of $X$ along $Z$.  Denote by  $E_\ast$ the exceptional divisor which is isomorphic to  $\PP^2$ and by $E'$ the proper transform of $E$. In this case, the exceptional divisor $E_\ast$ has two strata, namely, $E^\circ_\ast =E_* \setminus E' \cong \LL^2$ and $L_1=E_* \cap E'  \cong \PP^1$. 
In particular, any finite connected unbranched cover of such a stratum is biregular to the stratum itself. And it can be easily seen that for each of the two strata of $E_*$, the corresponding branched covers have exactly $m$ connected components. Thus, 
the contributions of $Z$ to $Z_{X,E}(T)$, and resp. of $E_\ast$ to $Z_{(X', E')}(T)$ are 
$$[\mu_{m}] \frac{A_1}{1 - A_1},  \quad {\rm resp.} \quad 
[\mu_{m}] \LL^2 \frac{A_\ast}{1 - A_\ast} + [\mu_{m}] (\LL^2 - 1) \frac{A_1}{1 - A_1} \frac{A_\ast}{1 - A_\ast}.$$
After we factor out $[\mu_m]$ and reduce to common denominator, it remains to show that 
$$A_1( 1 - A_\ast) = \LL^2 (1 - A_1) A_\ast + (\LL^2 - 1) A_1 A_\ast,$$
or equivalently $A_1 = \LL^2 A_\ast$, which follows from $m_\ast = m$ and $\nu_\ast = \nu_1 + 2$ (here $\codim_{E_I}Z=2$). 
\end{example}

Let us now prove the beginning  case of induction for Proposition \ref{ind}. 
\begin{proposition} \label{pt}
The assertion of Proposition \ref{ind} holds in the case when the center of blow-up $Z$ is zero-dimensional.
\end{proposition}

It suffices to prove Proposition \ref{pt} in the case when the center of blow-up is a single point. Indeed, the blow-up at a finite number of points can be regarded as a finite number of single point blow-ups. 

We can thus assume that $Z$ is a point. Let  $r+1=\codim_X Z$, which, by our assumption, equals $\dim X$. Then the exceptional divisor is $E_* \cong \mathbb{P}^r$.   The divisor $E=\sum_{i \in J} E_i$ of $X$ transforms under the blow-up into the divisor 
$E'$ in $X'$ consisting of  the proper transforms $E'_i$ of the irreducible 
components $E_i$ of  $E$,  
together with the exceptional component $E_*$. 
It suffices to restrict the comparison of  zeta functions $Z_{X,E}(T)$ and $Z_{X',E'}(T)$ only to contributions coming from the strata in the center of blow-up and the exceptional divisor, respectively.

As in the above examples, we need to describe the stratification of $E_\ast\cong \PP^r$ induced from that of $E'$ (see (\ref{str}) for the latter). Assume that $$Z \subseteq \bigcap_{i=1}^k E_i.$$ We recall here the following result from \cite{GVLM}.

\begin{lemma}\label{strat} \cite[Lemma 3.13]{GVLM} For each $k$ with $1 \leq k \leq r+1$ we have the following identity in $K_0({\rm Var}_\CC)$:
\begin{equation}\label{excdivblowuppt}[\mathbb{P}^r] = \sum_{l=0}^{k-1} \binom{k}{l} \mathbb{L}^{r-k+1} (\mathbb{L}-1)^{k-l-1} + [\mathbb{P}^{r-k}].\end{equation}
 The right-hand side describes the stratification of the exceptional divisor $E_* \cong \mathbb{P}^r$ induced by the divisor $\sum_{i=1}^k E'_i$ consisting of the proper transforms of components of $E$ containing the center of blow-up. More precisely, by setting  $I:=\{1,\cdots, k\}$, the strata of $E_*$ are:
 \begin{itemize}
\item  $L_{I}:=(\bigcap_{i=1}^k E'_i) \cap E_*$, which is isomorphic to $\mathbb{P}^{r-k}$. 
\item $\binom{k}{l}$ strata of dimension $r-l$ and of the form $$L_G:=(\bigcap_{i \in G} E'_i) \cap E_*\setminus \bigcup_{i \in I \setminus G} E'_i,$$ with $G \subset I$ and $1 \leq |G|=l \leq k-1$, each of which is isomorphic to $\CC^{r-k+1} \times (\CC^*)^{k-l-1}$.
The  class of each such stratum in $K_0({\rm Var}_\CC)$ is equal to $\mathbb{L}^{r-k+1} (\mathbb{L}-1)^{k-l-1}$.
\item $E^{\circ}_{\ast}:=E_* \setminus \bigcup_{i =1}^k E'_i$, of dimension $r$, which is isomorphic to $\CC^{r-k+1}\times(\CC^*)^{k-1}$, and whose class in $K_0({\rm Var}_\CC)$ is $\LL^{r-k+1}(\LL-1)^{k-1}$.
\end{itemize}
\end{lemma}

We also need the following easy fact:
\begin{lemma}\label{auxlemA}For a ring $R$ with unit containing the elements $A_1, \dots, A_k$ and $1 + A_1, \dots, 1+A_k$,  for any $k \in \NN$ we have that 
\begin{equation}\label{auxAi}\prod_{i=1}^k (1 - A_i) + \sum_{l=1}^k \sum_{{\substack{ G \subseteq \{1,2, \dots, k\} \\ |G| = l}}} \prod_{i\in G} A_i \prod_{j \in \{1,2, \dots, k\} \setminus G} (1 - A_j)= 1\in R\end{equation}
\end{lemma}

\begin{proof} This follows iteratively from the fact that $(1 -A_1) + A_1 = 1$ and, if we denote by  $B(k)$  the left-hand side of (\ref{auxAi}) then we have that 
$$B(l)(1 - A_{l+1}) + B(l) A_{l+1} =  B(l+1).$$ \end{proof}

\begin{proof} {(of Proposition \ref{pt})}\newline
As already pointed out, it suffices to check  the invariance (\ref{S=S}) of motivic infinite cyclic zeta function under blow-up  in the case when $Z$ is a single point. Assume $Z \subseteq \bigcap_{i=1}^k E_i$.  
Let $I=\{1, 2, \dots, k\}$ and set $m = {\gcd}(m_1,\cdots,m_k)$, where the $m_i$ are the values of the holonomy on the meridians $\delta_i$ about the components $E_i$. Clearly, the class $[\widetilde{E}_I^\circ{}_{|_Z}, \sigma_\Delta]$ equals $[\mu_m]$. Therefore, the corresponding contribution of $Z$ to the  left hand-side of (\ref{S=S}) is 
$$[\mu_m](\mathbb{L}-1)^{k-1} \prod_{i=1}^k \frac{\LL^{-\nu_i}T^{m_i}}{1 - \LL^{-\nu_i}T^{m_i}} .$$

On the other hand, for any stratum $S$ of the exceptional divisor $E_*$ (as described in Lemma \ref{strat}), the motive of the corresponding unbranched cover of Definition \ref{ucov} can be computed by $[\widetilde{S},\sigma_{\Delta'}]=[\mu_m][S,\sigma_{\Delta}]$, see the proof of Proposition 3.12 in \cite{GVLM} for more details. 

Taking into account the description of the stratification in Lemma \ref{strat} the contribution of $E_\ast$ to the right hand-side of (\ref{S=S}) is
\begin{multline*}[\mu_m] (\LL-1)^{k-1} \frac{A_\ast}{1 - A_\ast}   \cdot  \left ( \LL^{r-k+1}  + \sum_{l=1}^{k-1} \mathbb{L}^{r-k+1}  \sum_{{\substack{G \subseteq I \\
 |G| = l}}}  \prod_{i \in G}  \frac{A_i}{1 - A_i}  +
[\PP^{r-k}] (\LL - 1) \prod_{i=1}^k  \frac{A_i}{1 - A_i}  \right ).\end{multline*}
The first summand in the above expression corresponds to the stratum $E^\circ_\ast$ of dimension $r$, the summands indexed by $l$ correspond to the ${k \choose l}$ strata $L_G$ of dimension $r-l$ (with $G \subset I$, $1 \leq | G | = l \leq k-1$), while the last summand corresponds to the stratum $L_I$ of dimension $r-k$. 

After we factor out $[\mu_m](\LL - 1)^{k-1}$ and we reduce to the common denominator $(1 -A_\ast) \prod_{i=1}^k (1 - A_i)$, we are left to verify that
$\prod_{i=1}^k A_i (1 - A_\ast)$ equals  
$$A_\ast \big [ \LL^{r-k+1}  \prod_{i=1}^k (1 - A_i) + \sum_{l=1}^{k-1} \LL^{r-k+1} \sum_{{\substack{I \subseteq K \\ |I| = l}}}  \prod_{i\in I} A_i \prod_{j \in K \setminus I} (1 - A_j) + (\LL^{r-k+1} - 1) \prod_{i=1}^k A_i \big ].$$ 
The last claim follows from Lemma \ref{auxlemA}. Indeed, it can be reduced to to $\prod_{i=1}^k A_i = \LL^{r-k+1}A_\ast$, which holds because $m_\ast = \sum_{i=1}^k m_i$ and $\nu_\ast = \sum_{i=1}^k \nu_i + r-k+1$. Recall that here $\codim_{E_I}Z=r-k+1$.
\end{proof}


\subsection{Invariance of the motivic infinite cyclic covers under blowups: general case}
The main ideas of the proof of the induction step are already visible in the following example.

\begin{example}\label{D}
Let $X$ be a threefold and $E=\sum_{i\in J} E_i$, with $J=\{1,2,3\}$, be a simple normal crossing divisor. Let $Z$ be the intersection of $E_1$ and $E_2$. Set $I=\{1,2\}$, so in the notations from the introduction, we have that $Z=E_I$. The component  $E_3$ is transversal to $Z$.
Let us consider the blow-up $X'=Bl_Z X$ of $X$ along the center $Z$. As before, we denote the exceptional divisor $\PP(\nu_Z)$ by $E_*$. 

The strata in $Z$ are $E_I^{\circ}=E_1\cap E_2 \setminus E_3$ and the point $E_J=\cap_{i\in J} E_i$, so the contribution of the center $Z$ to the motivic infinite cyclic zeta function $Z_{X,E}(T)$ is: $$[\widetilde{E^{\circ}_I}] (\LL -1)\frac{A_1}{1 - A_1}\frac{A_2}{1 - A_2} + [\widetilde{E_J}](\LL -1)^2\frac{A_1}{1 - A_1}\frac{A_2}{1 - A_2}\frac{A_3}{1 - A_3}.$$
The exceptional divisor $E_*$ acquires a stratification with strata of the form:
$$L_H:=(\bigcap_{i \in H} E'_i) \cap E_*\setminus \bigcup_{i \in J \setminus H} E'_i,$$ with $H \subset J$, where the dense open stratum $E_*^{\circ}$ in $E_*$ is identified with $L_{\emptyset}$. More precisely, the strata of $E_*$ are in this case the following:
\begin{itemize}
\item $L_{\{1,3\}}=E_* \cap E'_1 \cap E'_3$, $L_{\{2,3\}}=E_* \cap E'_2 \cap E'_3$.
\item  $L_{\{1\}}=(E_* \cap E'_1) \setminus E'_3$, $L_{\{2\}}=(E_* \cap E'_2) \setminus E'_3$, $L_{\{3\}}=(E_* \cap E'_3) \setminus (E'_1 \cup E'_2)$. 
\item $E_*^{\circ}=E_*\setminus \bigcup_{i=1}^3 E'_i$.
\end{itemize}
So the contribution of the exceptional divisor $E_*$ to $Z_{X',E'}(T)$ is:
\begin{equation*}\begin{split}
\frac{A_\ast}{1 - A_\ast} 
& \left( 
[\widetilde{E_*^{\circ}}] +[\widetilde{L_{\{1\}}}](\LL-1)\frac{A_1}{1 - A_1} + [\widetilde{L_{\{2\}}}](\LL-1)\frac{A_2}{1 - A_2}  \right . \\
& + [\widetilde{L_{\{3\}}}] (\LL-1)\frac{A_3}{1 - A_3} + [\widetilde{L_{\{1,3\}}}](\LL-1)^2 \frac{A_1}{1 - A_1}\frac{A_3}{1 - A_3} \\
& \left . + [\widetilde{L_{\{2,3\}}}]\LL-1)^2 \frac{A_2}{1 - A_2}\frac{A_3}{1 - A_3} 
\right) .
\end{split}
\end{equation*}
Note that by Example \ref{A}, applied to the blow-up of the point $E_J$ of intersection of transversal curves $E_1 \cap E_3$ and $E_2 \cap E_3$ inside the surface $E_3$, we have that:
\begin{multline*}[\widetilde{E_J}](\LL -1)\frac{A_1}{1 - A_1}\frac{A_2}{1 - A_2}  =  \left ( [\widetilde{L_{\{3\}}}]  +  \big( [\widetilde{L_{\{1,3\}}}]\frac{A_1}{1 - A_1}+[\widetilde{L_{\{2,3\}}}] \frac{A_2}{1 - A_2} \big) (\LL -1) \right ) \frac{A_\ast}{1 - A_\ast} .\end{multline*}
In particular, this relation is consistent with the fact that $m_\ast= m_1 + m_2$ and $\nu_\ast = \nu_1 + \nu_2$. 

So in order to show that the contributions of $Z$ and $E_*$ to the zeta functions  $Z_{X,E}(T)$ and respectively $Z_{X',E'}(T)$ coincide, it suffices to prove the equality:
\begin{multline*}\label{pr}[\widetilde{E^{\circ}_I}] (\LL -1)\frac{A_1}{1 - A_1}\frac{A_2}{1 - A_2}  = \left( [\widetilde{E_*^{\circ}}] + ( [\widetilde{L_{\{1\}}}]\frac{A_1}{1 - A_1} +[\widetilde{L_{\{2\}}}]\frac{A_2}{1 - A_2}) (\LL-1) \right) \frac{A_\ast}{1 - A_\ast}.\end{multline*}
Next, note that by the definition of blow-up, we have isomorphisms $L_{\{1\}} \cong E^{\circ}_I \cong L_{\{2\}}$ which, moreover,  extend (by Lemma 3.2 in \cite{GVLM}) to isomorphisms between the corresponding unbranched covers appearing in the definition of our motivic zeta function. Also, by Lemma 3.1 of \cite{GVLM}, the (Zariski) locally trivial fibration $E_*^{\circ} \to E^{\circ}_I$ (with fiber $\PP^1 \setminus \{2 \ {\rm points}\}=\CC^*$) can be lifted to a $\CC^*$-fibration $\widetilde{E_*^{\circ}} \to \widetilde{E^{\circ}_I}$. Finally, the Zariski triviality implies that $[\widetilde{E_*^{\circ}},\sigma_{\Delta'}] = [ \widetilde{E^{\circ}_I},\sigma_{\Delta}] (\LL-1)$, which reduces the above claim to checking that
\begin{equation}\label{pr}[\widetilde{E^{\circ}_I}] (\LL -1)\frac{A_1}{1 - A_1}\frac{A_2}{1 - A_2}  = [\widetilde{E^{\circ}_I}] (\LL -1)\frac{A_\ast}{1 - A_\ast} \left( 1 + \frac{A_1}{1 - A_1} + \frac{A_2}{1 - A_2} \right). \end{equation}
After we factor out $[\widetilde{E^{\circ}_I},\sigma_{\Delta}] (\LL -1)$ and reduce to common denominator, this amounts to
$$A_1 A_2 ( 1 - A_\ast) = A_\ast (1 - A_1) (1 - A_2) + A_\ast A_1 (1- A_2) + A_\ast (1 - A_1) A_2,$$
or equivalently, $A_1 A_2 = A_\ast$. The latter equality  follows from   $m_\ast= m_1 + m_2$ and $\nu_\ast = \nu_1 + \nu_2$. \hfill$\square$
\end{example}


\begin{proof} {(of Proposition \ref{ind})}\newline 
The proof below is an enhanced version of (and relies on) that of \cite[Proposition 3.8]{GVLM}, and we will use freely facts proved in loc.cit. 

\smallskip

As already mentioned, the proof of Proposition \ref{ind} is by induction on the dimension of the center of blow-up. The beginning of induction (i.e., the case of one point) is proved in Proposition \ref{pt}. Note that, in general, the center of blow-up $Z$ is either contained in a component $E_i$ of $E$, or it is transversal to it, or it doesn't intersect it at all. We refer to components of the second kind as {\it transversal} components of $E$ (with respect to $Z$). By collecting all indices $i$ of components of $E$ containing $Z$, we note that  the center $Z$ is contained in a set $E_I$ (for some $I \subseteq J$) given by intersections of components of the deleted divisor. In particular, $Z$ gets an induced stratification from that of $E_I$. So, there is a dense open stratum $Z \cap E^{\circ}_I$ in $Z$, together with positive codimension strata obtained by intersecting $Z$ with collections of transversal components.

\smallskip

We begin the proof by first studying the case when the center of blowup is of type $E_I$, for some $I \subseteq J$. 

Let $X'$ be the blowup of $X$ along the center $Z$ defined as the intersection $E_I:=\bigcap_{i=1}^k E_i$ of some of the irreducible components of the deleted divisor $E$ (so $I=\{1,2,\cdots,k\}$), and also assume that the irreducible components $E_j$ for $j= k+1, \dots, \ell$ of $E$ intersect the center $Z$ transversally, and no other components of $E$ intersect  $Z$. In this case, $Z$ is stratified by a top dimensional open dense stratum $E^\circ_I$, and by positive codimension strata obtained by intersecting $Z$ with some of the transversal components $E_j$ (with $j= k+1, \dots, \ell$), i.e., strata of the form $E^\circ_{I \cup K}$, where $K \ne \emptyset$ and $K \subseteq \{k+1, \dots, \ell\}$. 
Therefore,  the contributions to the motivic infinite cyclic zeta function $Z_{X,E}(T)$ supported on $Z$ are
\begin{equation}\label{lhs}
\prod_{i=1}^k \frac{A_i}{1 - A_i} \left ( [\widetilde{E^\circ_I}](\mathbb{L} -1)^{k-1} + \sum_{{\substack{K \subseteq \{k+1, \dots, \ell\} \\
 K\ne \emptyset}}}[\widetilde{E^\circ_{I \cup K}}](\mathbb{L} -1)^{k+ |K| -1} \prod_{h \in K} \frac{A_h}{1 - A_h}\right ).
\end{equation}

After blowing up $X$ along $Z$, we get the deleted divisor $$E' = (\bigcup_{j\in J} E'_j) \cup E_\ast$$ of $X'=Bl_Z X$, where $E_\ast$ is the exceptional locus of the blow-up and $E'_j$ is the proper transform of $E_j$ (for $j \in J$). Note that, by the choice of the center $Z$ of blow-up, the $k$-fold intersection of the proper transforms of components $E_i$ with $i=1, \dots, k$ is empty,  i.e., $$ \bigcap_{i=1}^k E'_i = \emptyset.$$ The exceptional divisor $E_\ast$ is stratified by the top dimensional open stratum $L_{\emptyset}=E^\circ_\ast$, by the codimension $s$ (for  $s <k$) strata obtained by intersecting $E_*$  with $s$-fold intersections of the components $E'_1, \dots, E'_k$ of $E'$, i.e., by the strata   $L_G$ with $G \subset I$ is a proper subset of $I$, and by  strata contained in intersections of the proper transforms $E'_j$ for $j = k+1, \dots, \ell$ of the transversal components, i.e., strata of the type  $L_{G \cup K}$ where $G \subset I$ ($G \neq I$) a (possibly empty) subset and $K$ is a nonempty subset of  $\{k+1, \dots, \ell\}$.  Therefore the contributions to the motivic infinite cyclic zeta function  $Z_{X', E'}(T)$ supported on $E_\ast$ are:

\begin{multline}\label{rhs}
\frac{A_\ast}{1 - A_\ast}  \left( [\widetilde{E^\circ_\ast}] +  \sum_{{\substack{G \subset I,\\ G \ne \emptyset, I}}} [\widetilde{L_G}](\mathbb{L}-1)^{|G|} \prod_{g \in G} \frac{A_g}{1 - A_g}  + \right .\\\left . \sum_{{\substack{G \subset I,\\ G \ne I}}}\sum_{{\substack{K \subseteq \{k+1, \dots, \ell\} \\
 K\ne \emptyset}}} [\widetilde{L_{G \cup K}}] (\mathbb{L}-1)^{|K|+|G|} \prod_{h \in G \cup K} \frac{A_h}{1 - A_h}   \right) .\end{multline}

We can now apply induction on the dimension of the center of blowup, and the exclusion-inclusion principle, to show that 
strata of the center $Z$ which are contained in intersections of the transversal components $E_j$, for $j=k+1,\cdots, \ell$, give equal contributions to the zeta functions $Z_{X, E}(T)$ and $Z_{X', E'}(T)$. More precisely,
for each positive codimension stratum $E^\circ_{I \cup K}$ (with $K \subseteq \{k+1, \dots, \ell\}, 
 K\ne \emptyset$) of $Z$, we get by induction for the blow-up of $E_K$ along the center $Z \cap E_K=E_{I \cup K}$, and with deletion divisor $E_K \cap (\sum_{i=1}^k E_i)$, a relation of the type
\begin{multline}\tag{$\ast_{K}$}\label{induc}
(\mathbb{L} -1)^{k -1} \prod_{i=1}^k \frac{A_i}{1-A_i} 
\cdot \left( [\widetilde{E^\circ_{I \cup K}}] + \sum_{{\substack{K \subset K' \subseteq \{k+1, \dots, \ell \} \\ K'\setminus K \neq \emptyset}}}  [\widetilde{E^{\circ}_{I \cup K'}}](\mathbb{L}-1)^{|K'\setminus K|} \prod_{h \in K'\setminus K} \frac{A_h}{1-A_h} \right) \\
= \frac{A_\ast}{1-A_\ast} \left ( [\widetilde{L_K}]  +  \sum_{{\substack{K \subset K' \subset \{k+1, \dots, \ell\} \\ K' \setminus K \neq \emptyset}}}[\widetilde{L_{K'}}](\mathbb{L}-1)^{|K' \setminus K|}  \prod_{h \in K' \setminus K} \frac{A_h}{1-A_h} \right . \\
 + \sum_{{\substack{G \subset I,\\ G \ne \emptyset, I}}} [\widetilde{L_{G \cup K}}](\mathbb{L}-1)^{|G|}  \prod_{g \in G} \frac{A_g}{1-A_g} \\
\left . + \sum_{{\substack{G \subset I,\\ G \ne \emptyset, I}}} \sum_{{\substack{K \subset K' \subset \{k+1, \dots, \ell\} \\ K' \setminus K \neq \emptyset}}}[\widetilde{L_{G \cup K'}}](\mathbb{L}-1)^{|G \cup (K' \setminus K)|}  \prod_{g \in G} \frac{A_g}{1-A_g}  \prod_{h \in K' \setminus K} \frac{A_h}{1-A_h} \right ) .\end{multline}
By summing up all the products (\ref{induc})$\cdot (\mathbb{L} -1)^{|K|} \prod_{h \in K} \frac{A_h}{1 - A_h}$ for the positive codimension strata $E^\circ_{I \cup K}$ of $Z$, we reduce the comparison of (\ref{lhs}) and (\ref{rhs}) to proving the identity:
\begin{equation}\label{finaleq}
[\widetilde{E^\circ_I}](\mathbb{L} -1)^{k-1}\prod_{i=1}^k \frac{A_i}{1-A_i} =\frac{A_\ast}{1-A_\ast} \left(  [\widetilde{E^\circ_\ast}]+ \sum_{{\substack{G \subset I,\\ G \ne \emptyset, I}}} [\widetilde{L_G}] (\mathbb{L}-1)^{|G|} \prod_{g \in G} \frac{A_g}{1-A_g}\right),
\end{equation}
i.e.,  it remains to show that the contribution of the dense open stratum of the center $Z$ to the zeta function $Z_{X, E}(T)$ coincides with the contribution to $Z_{X', E'}(T)$ of the strata supported on the exceptional divisor $E_*$ and which are not contained in the proper transforms of the transversal components of $E$. 

Note that in the proof of Proposition 3.8 in \cite{GVLM} we have shown that, for any subset $G \subset I=\{ 1, \dots, k\}$, $G \neq I$ (but including the empty set corresponding to  $L_\emptyset = E^\circ_\ast$), we have the following identification:
$$[\widetilde{L_G}] = [ \widetilde{E^{\circ}_I}] (\LL-1)^{k - |G| - 1}.$$
Therefore we can factor out $[\widetilde{E^{\circ}_I}] (\LL-1)^{k - 1}$  in (\ref{finaleq}) and we are left to check that
$$\prod_{i=1}^k \frac{A_i}{1-A_i} =  \frac{A_\ast}{1-A_\ast} \left( 1+ \sum_{{\substack{G \subset I,\\ G \ne \emptyset, I}}}  \prod_{g \in G} \frac{A_g}{1-A_g} \right).$$
The last claim follows from Lemma \ref{auxlemA}.  Indeed, it reduces to $\prod_{i=1}^k A_i=A_{\ast}$, which holds since $m_*=\sum_{i \in I} m_i$ and $\nu_*=\sum_{i \in I} \nu_i$. (Here $\codim_{E_I}Z=0$.)

\medskip

Let us now explain the proof in the general case, i.e., when  the center $Z$ is strictly contained in some set $E_I$, for $I \subseteq J$, and let $I=\{1,\cdots,k\}$. Assume that the codimension of $Z$ in $X$ is $d+1\geq k$. Therefore the codimension of $Z$ in $E_I$ is $d+1-k$. Again, by induction, it suffices to show   that the contribution of the dense open stratum $Z^{\circ}:=Z \cap E_I^{\circ}$ of the center $Z$ to the zeta function $Z_{X, E}(T)$ coincides with the contribution to $Z_{X', E'}(T)$ of the strata supported on the exceptional divisor $E_*=\mathbb{P}(\nu_Z)$ which are not contained in the proper transforms of the transversal components components of $E$ (with respect to $Z$), that is, 
\begin{equation}\label{finaleq2}
[\widetilde{Z}^\circ](\mathbb{L} -1)^{k-1}\prod_{i=1}^k \frac{A_i}{1 - A_i} =\frac{A_\ast}{1 - A_\ast} \left( [\widetilde{E^\circ_\ast}]  + \sum_{{\substack{G \subseteq I,\\ G \ne \emptyset}}} [\widetilde{L_G}](\mathbb{L}-1)^{|G|} \prod_{g \in G} \frac{ A_g}{1 - A_g} \right).
\end{equation}
On the right hand side of (\ref{finaleq2}), we use the same notation as before for the stratification of the exceptional divisor $E_*$. Note that in this case we have to also allow $G=I$ in the sum of the right hand side term of (\ref{finaleq2}) because $Z$ is strictly contained in $E_I$, and therefore $\cap_{i=1}^k E'_i \neq \emptyset$. 

Recall that in the proof of Proposition 3.8 in \cite{GVLM} we have showed that,   for any subset $G \subset I = \{1, \dots, k\}$, $G \neq I$ (but {including the empty set} corresponding to $L_\emptyset = E^\circ_\ast$), we have:
\begin{equation}\label{20a} [\widetilde{L_G}] = [ \widetilde{Z^{\circ}}] \LL^{d-k+1}(\LL -1)^{k - |G| - 1}.\end{equation}
Moreover, for $G=I$, we have proved that 
\begin{equation}\label{20b} [\widetilde{L_I}] = [\widetilde{Z^\circ}](\LL^{d-k} + \LL^{d-k-1}+\cdots + \LL + 1).\end{equation}
By substituting the equalities (\ref{20a}) and (\ref{20b}) into  (\ref{finaleq2}), and factoring out $[\widetilde{Z}^\circ](\LL-1)^{k-1}$, it remains to show that:
\begin{multline}\label{finaleq3}
\prod_{i=1}^k \frac{A_i}{1 - A_i}  = \frac{A_\ast}{1 - A_\ast} \left (  \LL^{d-k+1} \Big( 1 + \sum_{{\substack{G \subset I,\\ G \ne \emptyset, I}}}  \prod_{g \in G} \frac{ A_g}{1 - A_g} \Big) + (\LL^{d-k+1} -1) \prod_{i=1}^k \frac{ A_i}{1 - A_i} \right).
\end{multline}
Reducing to the common denominator $\prod_{i=1}^k (1-A_i) (1 - A_\ast)$, it is then enough to check that
\begin{multline*}\prod_{i=1}^k {A_i}(1 - A_\ast)  = {A_\ast}  \LL^{d-k+1} \left (  \prod_{i=1}^k (1 - A_i) + \sum_{{\substack{ G \subseteq I \\  G \ne \emptyset}}} \prod_{g\in G} A_g \prod_{j \in I \setminus G} (1 - A_j) \right ) - A_\ast \prod_{i=1}^k \ A_i.\end{multline*}
By using Lemma \ref{auxlemA}, this can be further reduced to showing that
$\prod_{i=1}^k A_i=A_{\ast} \LL^{d-k+1}$, which holds since
  $m_\ast= m_1 + \cdots + m_k $ and $\nu_\ast = \nu_1 + \cdots + \nu_k + d - k + 1$. (Here $\codim_{E_I}Z=d-k+1$). \end{proof}

\begin{remark}The proof of  Proposition \ref{ind} works whenever (\ref{defnmzf}) and (\ref{afterbl}) are defined. Remarks \ref{rmk1} and \ref{rafterbi} explain that it is sufficient (although not necessary) to assume that the holonomy takes strictly positive values. An alternative sufficient hypothesis is that either $\nu_i \ne 0$ or $m_i \ne 0$ for all $i \in J \cup \{ \ast \}$.  The rest of the proof does not assume that the holonomy map takes always strictly positive values. In fact, the proof shows  that the expression defining $Z^A_{X,E,\Delta, \nu}(T)$ in (\ref{defnmzf}) is well-defined in the sense of the Definition \ref{equivnbhdld} if and only if the holomony and log-discrepancy values satisfy the correct blow-up relations given in (\ref{blowupholomonyeqn}) and (\ref{eleprocld}) and never vanish simultaneously. 

As a consequence, Theorem \ref{Welldefined} holds as long as either $\nu_i \ne 0$ or $m_i \ne 0$ for all irreducible components of $E$ and their transforms along any sequence of blow-ups and blow-downs.  
  
In  \cite{V2}, Veys has used the condition that either $\nu_i \ne 0$ or $m_i \ne 0$ to introduce {\it naive motivic zeta functions} in the cases of  a Kawamata log terminal and divisorial log terminal (dlt) pair $(X,E)$ and in the case that the pair $(X,E)$ has no strictly log canonical singularities assuming the existence of a log relative minimal model of $(X,E)$ which is dlt.   See Section \ref{ns}, in particular (\ref{naivedlzf}) and (\ref{naivezf}),  for an explanation of the adjective ``naive''.\end{remark}


\section{Relation with the Denef-Loeser motivic zeta function}\label{relationwithmotivic}


Denef and Loeser introduced the {\it local motivic zeta function} at a point $x$ for a non-constant morphism $f: \mathbb{C}^{d+1} \rightarrow \mathbb{C}$  with $f(x)=0$ (e.g., see \cite[Definition 3.2.1]{Barcelona} and the references therein) as

\begin{equation}\label{blaa}
Z_{f,x}(T):=\sum_{n\geq 1} [\mathcal{X}_{n,1}]\LL^{-(d+1)n}T^n \in \mathcal{M}^{\hat{\mu}}_{\CC}[[T]],\end{equation} where $\mathcal{X}_{n,1}$ denotes the set of $(n+1)$-jets $\varphi$ of $\mathbb{C}^{d+1}$ such that $f \circ \varphi = t^n + \dots$. 
Note that there is a good action of the group $\mu_n$ (hence of $\hat\mu$) on $\mathcal{X}_{n,1}$ by $\lambda \times \varphi \mapsto \varphi( \lambda \cdot t)$. Moreover, Denef and Loeser \cite[Theorem 3.3.1]{Barcelona} proved  with the help of an embedded resolution of the singularities of $f$  that  the motivic zeta function $Z_{f,x}(T)$ belongs to $\mathcal{M}^{\hat{\mu}}_{\CC}[[T]]_{{\rm sr}}$. Then they also defined  the {\it local motivic Milnor fibre} $\mathcal{S}_{f,x}$  (e.g., see \cite[Definition 3.5.3]{Barcelona} and the references therein) as a limit in the sense of (\ref{deflim}):  
\begin{equation}\label{bla} \mathcal{S}_{f,x}: = - \lim_{T \to + \infty} Z_{f,x}(T) \in \mathcal{M}^{\hat{\mu}}_{\CC}, \end{equation}
see also \cite[Section 2.8]{GLM} and \cite[Lemma 4.1.1]{DenefLoeserJAG}.

In \cite[Theorem 5.1]{GVLM} we related the concepts of local motivic Milnor fibre $\mathcal{S}_{f,x}$ and the motivic infinite cyclic cover $S_{X,E, \Delta}$. The following result extends this relation to the  realm of motivic zeta function and the motivic infinite cyclic zeta function, respectively. 
\begin{theorem}\label{milnorfiberinfcover}
Let $f: \mathbb{C}^{d+1} \rightarrow \mathbb{C}$ be a 
non constant morphism with $f(x)=0$,  and 
 $p : X \rightarrow \mathbb{C}^{d+1}$ 
be a log-resolution of the singularities of the pair 
$(\CC^{d+1}, f^{-1}(0))$. 
Choose $p$ so that $(p^{-1}(x))_{{\rm red}}$ is a union of components of $(p^{-1}(f^{-1}(0)))_{{\rm red}}$. Let $E = \sum_{j \in J} E_j$ be the decomposition  of $(p^{-1}(f^{-1}(0)))_{{\rm red}}$ into irreducible components, and  let $A = \{i \in J \, | \, E_i \subset p^{-1}(x)\}$. 
Then the following hold:
\begin{enumerate}
\item[(1)] For $\epsilon > 0$ small enough, and $B(x,\epsilon)$ a ball of radius $\epsilon$ centered at $x \in \CC^{d+1}$, the resolution map $p$ provides a biholomorpic identification between  
$B(x,\epsilon)\setminus \{f=0\}$ and $T^\ast_{E^A}$, the punctured regular neighborhood of the divisor $E^A:=\sum_{i \in A} E_i$. 
In particular, the map $\gamma \rightarrow \int_{\gamma} {df \over f}$ 
can be viewed as a holonomy homomorphism 
$\Delta: \pi_1(T^\ast_{E^A}) \rightarrow \mathbb{Z}$  
of the punctured neighborhood of $E^A$. This holonomy map   
takes the boundary $\delta_i$ of any small disk transversal 
to the irreducible component $E_i$ of $E^A$ to the multiplicity $m_i$ of $E_i$ in the divisor of $f \circ p$, i.e., $\Delta(\delta_i)=m_i$ for all $i \in A$.  
\item[(2)] The map taking the boundary $\delta_i$ of any small disk transversal to the irreducible component $E_i$ of $E^A$, to the multiplicity of $E_i$ in the divisor of the pullback of the standard volume form $\omega = { d} x_1 \cdots { d} x_{d+1}$ plus one, defines a log-discrepancy $\nu$, i.e., $\nu_i -1 = {\rm ord}_{E_i} p^\ast \omega$.
\item[(3)] One has the identity in $\mathcal{M}^{\hat{\mu}}_\mathbb{C}[[T]]_{{\rm sr}}$:
$$Z_{f,x}(T) = Z^A_{X,E,\Delta, \nu}(T).$$ 
\end{enumerate}
\end{theorem}

\begin{proof} The result follows from Theorem 5.1 in \cite{GVLM} and \cite[Section 2.8]{GLM}, see also  \cite[Lemma 4.1.1]{DenefLoeserJAG}.\end{proof}

\begin{remark} Note that Theorems \ref{milnorfiberinfcover} and \ref{Welldefined} give a direct argument of the fact that the right-hand side of the formula in \cite[Theorem 3.3.1]{Barcelona} expressing the motivic zeta function in terms of the data of a log resolution  is in fact independent of the choice of log resolution. This was a priori known only because of the relation (\ref{blaa}) defining the motivic zeta function in terms of arc spaces (which are intrinsic invariants), see also the discussion in \cite[Section 3.5]{Barcelona}. Moreover, we do not make use of the third relation in the definition of the equivariant Grothendieck ring for proving our result.
\end{remark}


\section{Further specializations of the motivic infinite cyclic zeta function}\label{ns}

In the previous section, we discussed the specialization of motivic infinite cyclic zeta function to the motivic infinite cyclic cover $S^A_{X,E,\Delta}$. The remarks after  Definition \ref{micc} explain how to recover $S^A_{X,E,\Delta}$ from $Z^A_{X,E, \Delta, \nu}(T)$. In our previous work we proved that in general the Betti realization of $S^A_{X,E,\Delta}$ is given (up to sign) by the cohomology with compact support of the infinite cyclic cover of finite type 
 $\widetilde{T}^\ast_{X,E,\Delta}$ of the punctured neighborhood $T^*_{X,E}$, see \cite[Proposition 4.2]{GVLM}. We also hinted that there should be a stronger version of this result involving the Hodge realization of $S^A_{X,E,\Delta}$, see \cite[Remark 4.3]{GVLM}. In this section, we discuss further specializations of our motivic infinite zeta functions.

For any  strictly positive integer $e$  and any character $\alpha$ of order $e$ we can use the equivariant Euler characteristic $\chi_{{\rm top}}( -, \alpha)$ to introduce a {\it twisted topological zeta function} as
$$\chi_{{\rm top}}\left((\mathbb{L}-1)Z^A_{X,E, \Delta, \nu}(\mathbb{L}^{-s}), \alpha \right).$$
One can also consider $p$-adic realizations as in \cite[Section 2.4]{DenefLoeserJAG}, see also \cite[Section 5.3]{N}.

One can also introduce a {\it naive motivic infinite cyclic zeta function} $Z^{{\rm naive}, A}_{X,E, \Delta, \nu}(T)$ as 
\begin{equation}\label{naivezf}\sum_{\substack{\emptyset \not = I \subseteq J\\ A \cap I \ne \emptyset}} [{E^\circ_I}](\mathbb{L}-1)^{|I|} \prod_{i \in I} \frac{\mathbb{L}^{-\nu_i}T^{m_i}}{1 - \mathbb{L}^{-\nu_i}T^{m_i}} \in \mathcal{M}_{\CC}[[T]]_{{\rm sr}}.\end{equation}
It is easy to check that under the hypothesis of Theorem \ref{milnorfiberinfcover} the zeta function $Z^{{\rm naive}, A}_{X,E, \Delta, \nu}(T)$ coincides with the naive motivic zeta function of Denef and Loeser, which is given by 
\begin{equation}\label{naivedlzf}Z^{{\rm naive}}_{f,x}(T):=\sum_{n\geq 1} [\mathcal{X}_{n}]\LL^{-(d+1)n}T^n \in \mathcal{M}_{\CC}[[T]],\end{equation} where $\mathcal{X}_{n}$ denotes the set of $(n+1)$-jets $\varphi$ of $\mathbb{C}^{d+1}$ such that $f \circ \varphi = at^n + \dots$, for some $a \ne 0$.

Hartmann \cite[Corollaries 8.5 and 8.6]{H1} has introduced a $\mathcal{M}_\CC$-linear  quotient map $\mathcal{M}^{{\rm \mu}}_\CC \rightarrow  \mathcal{M}_\CC$ that allows to recover  $Z^{{\rm naive}, A}_{X,E, \Delta, \nu}(T)$  from  $Z^A_{X,E, \Delta, \nu}(T)$. The naive motivic infinite cyclic cover equals $$(\mathbb{L}-1)Z^A_{X,E, \Delta, \nu}(T)/ \hat{\mu} \in \mathcal{M}_\mathbb{C},$$  
where $Z^A_{X,E, \Delta, \nu}(T)/ \hat{\mu}$ denotes the image under the extension to $\mathcal{M}^{\hat{\mu}}_\mathbb{C}[[T]] \rightarrow \mathcal{M}_\mathbb{C}[[T]]$ of the above mentioned map. See also \cite[Remarks 7.1 and 7.2]{H2}.

There are two further specializations of the naive infinite cyclic motivic zeta function worth mentioning. Firstly, there is a {\it Hodge infinite cyclic cover zeta function} defined as
$$\mathcal{H}^A_{X, E, \Delta, \nu}(T):=\sum_{\substack{\emptyset \not = I \subseteq J\\ A \cap I \ne \emptyset}} H({E^\circ_I}; u,v) \prod_{i \in I} \frac{(uv -1)T^{m_i}}{(uv)^{\nu_i}-T^{m_i}} \in \mathbb{Q}(u,v)[[T]],$$
where $H(-;u,v)$ is the Hodge-Deligne polynomial. Compare with \cite[Section 2.2]{SV}. 
Secondly, there is the topological Euler realization $Z^{{\rm top}, A}_{X,E, \Delta, \nu}(s)$ given by 
$$\chi_{{\rm top}}(Z^{{\rm naive}, A}_{X,E, \Delta, \nu}(\LL^{-s})) = \sum_{\substack{\emptyset \not = I \subseteq J\\ A \cap I \ne \emptyset}} \chi_{{\rm top}}(E^\circ_I) \prod_{i \in I} \frac{1}{\nu_i + s m_i} \in \QQ(s),$$
where $\chi_{{\rm top}}(-)$ denotes the topological Euler characteristic. It is again immediate to check that in the case  of Theorem \ref{milnorfiberinfcover} this coincides with the topological zeta function introduced by Denef and Loeser. 

Finally, Veys \cite{V2} has investigated how to use the theory of motivic zeta functions to generalize the theory of $E$-stringy invariants of Batyrev beyond the log terminal case.  Schepers and Veys showed in \cite[Proposition 2.3]{SV}, as a consequence of the inversion of the adjunction of Ein, Musta\c{t}a and Yasuda \cite[Theorem 1.6]{EMY}, that under the hypothesis of Theorem \ref{milnorfiberinfcover} and assuming that $f^{-1}(0)$ is irreducible normal (and canonical), 
the local stringy $E$-function of the hypersurface $f^{-1}(0)$ can be recovered as a {\it residue} of the local Hodge zeta function, i.e., the local stringy $E$-function can be obtained as: 
$$-\frac{1}{uv(uv-1)}\mathcal{H}^A_{X, E, \Delta, \nu}(T)(T-uv)|_{T=uv}.$$


\section{Motivic zeta functions for arbitrary divisors. Relation to the monodromy conjecture}\label{s6}

The framework of the zeta functions associated to infinite cyclic covers seems to be useful for the study of the monodromy conjecture and some of its generalizations (like for example the generalized monodromy conjecture proposed by N\'emethi and Veys in \cite{V, NV1, NV2}).

The monodromy conjecture predicts a relation between the poles of the local zeta function $Z_{f,x}(T)$ and the eigenvalues of local monodromies (i.e., monodromies of local Milnor fibers at points) of the hypersurface defined by $f$.  However, it is worth to notice here that the Grothendieck ring is not a domain and therefore one should be careful when speaking of  poles.  The {\it Igusa-Denef-Loeser  motivic monodromy conjecture} can be formulated as follows (see for example \cite[Remark at the end of Section 2.4]{DenefLoeserJAG} and \cite[Section 3.4.1]{Loe}).

\begin{conjecture}\label{IDLconj} There exists a finite subset $S$ of $\ZZ_{>0} \times \ZZ_{>0}$ such that 
$$Z_{f,x}(T) \in \mathcal{M}^{\hat{\mu}}_\mathbb{C}\left [ T, \frac{1}{1 - \LL^{-a}T^b} \right ]_{(a,b) \in S}$$ and such that for each $(a,b) \in S$, the value $\exp(-2 \pi i b/a)$ is a monodromy eigenvalue of $f$ at some point of $f^{-1}(0)$. 
\end{conjecture}

There are other, naive (see \cite[Section 6.8]{V1}), topological (see \cite[Conjecture 3.3.2]{DL}) or $p$-adic (see \cite{I}) versions of the monodromy conjecture that come from considering  different specializations of the zeta function, like the ones mentioned in the previous section.  Naive, topological and/or $p$-adic versions  of the monodromy conjecture have been proven in, among others, the case of curves by Loeser, certain surfaces by Veys, Rodrigues, Lemahieu, Van Proeyen, Bories, superisolated singularities and quasi-ordinary singularities by Artal-Bartolo, Cassou-Nogu\'es, Luengo and Melle-Hern\'andez, and hyperplane arrangements by Budur,  Musta\c{t}a and Teitler. The Igusa-Denef-Loeser {motivic monodromy conjecture} is different from the naive one as explained in \cite[Section 4 and Subsection 5.2]{Cau}. Nicaise and Veys have informed  us \cite{NVprivate} that the Igusa-Denef-Loeser {motivic monodromy conjecture} can be checked in the case of plane curves (though this fact is not available in the literature). 

One of the first difficulties in the study of the monodromy conjecture is the computation of the set $S$, whose elements are called {\it poles}. Explicit formulae for the zeta function in terms of an embedded resolution of singularities (see for instance \cite[Theorem 3.3.1]{Barcelona}) show that any exceptional divisor gives a candidate pole $(a, b)$.  But concrete computations also show the existence of many cancellations and that most irreducible exceptional divisors do not contribute to $S$. 

On the other hand, typically only a few eigenvalues are recovered from the set $S$. Moreover,  those eigenvalues coming from $S$ cannot be determined before-hand.  

N\'emethi and Veys have proposed an alternative approach in \cite{V, NV1, NV2} to study the topological Monodromy Conjecture in the case of singularities of plane curves. By considering a finite family of {\it allowed} differential forms associated to $f$ (see \cite[Section 1.5.3 and Definition 4.1.1]{NV2} for more details), they obtain a finite family of topological zeta functions such that their corresponding sets of poles recover all (and only) the eigenvalues of $f$ at $x$. This works because, while the first coordinate of each element of $S$ is totally determined by $f$, the second coordinate depends on the differential form. Moreover, N\'emethi and Veys show that arbitrary (i.e., not allowed) differential forms may produce poles $(a, b')$ such that $\exp(-2 \pi i b'/a)$ is not an eigenvalue of the local monodromies of $f^{-1}(0)$. 

Cauwbergs has extended the approach of N\'emethi and Veys in the case of naive motivic zeta function \cite[Corollary 5.2]{Cau}. However, he has also shown in \cite[Section 5.2]{Cau}, using twisted topological zeta functions as specializations of the (non naive) motivic zeta functions, that there is no set of allowed forms able to recover all eigenvalues from the poles of the (non naive) motivic zeta function. 
 
\smallskip

Let us conclude this note by suggesting a version of the monodromy conjecture in the case of the infinite cyclic covers, as considered in this paper. Let $D$ be any divisor on a  smooth quasi-projective variety $Y$, with $T^*_{Y,D}$ a punctured neighborhood of $D$ in $Y$. Under suitable assumptions, we can associate a motivic zeta function to the pair $(Y,D)$ as follows. Fix, as before, a holonomy map $\Delta:\pi_1(T^*_{Y,D}) \to \mathbb{Z}$, and let $\widetilde{T}^*_{Y,D}$ be the corresponding infinite cyclic cover. Let $(X,E)$ be a log resolution of $(Y,D)$, with log discrepancies $\nu_i$, $i \in J$, defined as usual in terms of the relative canonical divisor of the resolution map. Assume moreover that the induced holonomy (denoted also by $\Delta$) on the punctured neighborhood $T^*_{X,E}$ of $E$ in $X$  makes the infinite cyclic cover  $\widetilde{T}^*_{X,E}$ (and hence  $\widetilde{T}^*_{Y,D}$) of finite type. Then the Alexander modules $H^i(\widetilde{T}^*_{Y,D};\mathbb{C})$ ($i \geq 0$) are torsion over the Laurent polynomial ring $\mathbb{C}[t^{\pm 1}]$, with the $t$-multiplication corresponding to the action of a generator of the covering group $\mathbb{Z}$ (e.g., see \cite[Proposition 2.4 and Remark 2.5]{GVLM}). The order of the torsion module $H^i(\widetilde{T}^*_{Y,D};\mathbb{C})$ is called the $i$-th Alexander polynomial of $T^*_{Y,D}$, and will be denoted by $\delta_i(t)$.

The monodromy zeta function $Z_{Y,D,\Delta}(T)$ of the infinite cyclic cover $\widetilde{T}^*_{Y,D}$ can then be defined by formula (\ref{defnmzf}) of Definition \ref{maindef}.  According to Theorem \ref{Welldefined}, this is a well-defined invariant of the punctured neighborhood of $D$ in $Y$.

It is then interesting to compare the poles of the motivic zeta function $Z_{Y,D,\Delta}(T)$ with the zeros of the Alexander polynomials $\delta_i(t)$ ($i \geq 0$)  of the punctured neighborhood $T^*_{Y,D}$ of $D$ in $Y$. It was already noted in \cite{GVLM} that the zeros of $\delta_i(t)$ ($i \geq 0$) are roots of unity. In fact, one would be tempted to believe that the following claim is true:

\smallskip

{\it If $(a,b)$ is a pole of the motivic zeta function $Z_{Y,D,\Delta}(T)$ in the sense of Conjecture \ref{IDLconj}, then $\exp(-2\pi i b/a)$ is a zero of some Alexander polynomial $\delta_i(t)$ ($i \geq 0$) of the punctured neighborhood of $D$.}

\smallskip

The motivation behind such a claim goes as follows. Let us assume that a linking number can be defined on the complement of $D$, e.g., $D$ is homologically trivial in a simply-connected algebraic manifold $Y$, and let $\Delta$ be the holonomy defined by the linking number homomorphism. If, moreover, $D$ is a Cartier divisor, then an argument involving a Mayer-Vietoris spectral sequence can be used to show that any root of an Alexander polynomial  $\delta_i(t)$ of the punctured neighborhood of $D$ in $Y$ is also an eigenvalue of some monodromy operator acting on the local Milnor fiber (with respect to a local defining equation of $D$) at some point on $D$. When coupled with the above claim (assuming it holds true), this fact would then imply a version of the Igusa-Denef-Loeser monodromy Conjecture \ref{IDLconj}. For a different variant of Conjecture \ref{IDLconj}, see \cite{Xu}.

We should also point out that in the case of plane curves the above claim is actually equivalent to the Conjecture \ref{IDLconj}, as can be easily seen by using a Mayer-Vietoris long exact sequence. Hence our claim is true in this case by \cite{NVprivate}, as already indicated above. Nevertheless, in general, the collection of zeros of the Alexander polynomials $\delta_i(t)$ ($i \geq 0$) of the punctured neighborhood is (stricly) contained in the set of eigenvalues of the local monodromies at points along the divisor $D$. So, in some sense,  the above claim complements the N\'emethi-Veys approach mentioned earlier.




\bibliographystyle{amsalpha}

\end{document}